\documentclass[a4paper]{amsart}
\newcommand{\ourtitle}[1]{\title{#1}\newcommand{\ourtitlex}{#1}}
\ourtitle{The stable 4-genus of alternating knots}
\usepackage[latin1]{inputenc}
\usepackage{graphicx,amsmath,amssymb,microtype,hyperref,xcolor}
\usepackage[nameinlink]{cleveref}
\crefname{subsection}{subsection}{subsections}
\Crefname{subsection}{Subsection}{Subsections}
\let\cref\Cref
\definecolor{darkblue}{RGB}{0,0,0}
\definecolor{gray}{RGB}{127,127,127}
\definecolor{darkred}{RGB}{160,0,0}
\definecolor{lightyellow}{RGB}{255,255,128}
\hypersetup{colorlinks={true},linkcolor={darkblue},citecolor={darkblue},filecolor={darkblue},urlcolor={darkblue},
pdfauthor={Sebastian Baader and Lukas Lewark},pdftitle={\ourtitlex}}
\newcommand{\qua}{\hskip 0.4em \ignorespaces}
\def\arxiv#1{\relax\ifhmode\unskip\qua\fi
\href{http://arxiv.org/abs/#1}%
{\tt arXiv:\penalty -100\unskip#1}}    

\def\MR#1{\relax\ifhmode\unskip\qua\fi
\href{http://www.ams.org/mathscinet-getitem?mr=#1}{\tt MR#1}}
\def\xox#1{\csname xx#1\endcsname}

\crefname{equation}{}{}

\newtheorem{theorem}{Theorem}
\newtheorem{prop}[theorem]{Proposition}
\newtheorem{lemma}[theorem]{Lemma}
\theoremstyle{remark}
\newtheorem{example}[theorem]{Examples}
\DeclareMathOperator{\sgn}{sgn}

\author{Sebastian Baader}
\address{Mathematisches Institut\\
Universit\"at Bern\\
Sidlerstrasse 5\\
3012 Bern \\
Switzerland}
\email{sebastian.baader@math.unibe.ch}
\author{Lukas Lewark}
\address{Mathematisches Institut\\
Universit\"at Bern\\
Sidlerstrasse 5\\
3012 Bern \\
Switzerland}
\email{lukas@lewark.de}
\begin{document}
\maketitle
\begin{abstract}
We show that the difference between the genus and the stable topological 4-genus of alternating knots is either zero or at least 1/3.
\end{abstract}
\section{Introduction}
The topological 4-genus $g_t(K)$ of a knot $K \subset S^3$ is defined by minimising the genus over all locally flat surfaces $\Sigma \subset B^4$ with boundary $\partial \Sigma=K$. Determining the topological 4-genus is a hard problem, in general. The only workable lower bound is Kauffman-Taylor's signature bound~\cite{KaTa}: $\sigma(K) \leq 2g_t(K)$.
A remarkable upper bound was recently derived by Feller~\cite{Fe}: $g_t(K) \leq \text{maxdeg}(\Delta_K(t))$, where $\Delta_K(t)$ denotes the Alexander polynomial of $K$. This can be seen as a generalisation of Freedman's disc theorem~\cite{freedman2,freedman}. Somewhat more is known when $g_t(K)$ is replaced by its stable version defined by Livingston~\cite{stable}:
$$\widehat{g}_t(K)=\lim_{n \to \infty} \frac{1}{n} g_t(K^n),$$
where $K^n$ denotes the $n$-times iterated connected sum of $K$. For example, $\widehat{g}_t(K)$ is strictly smaller than the classical genus $g(K)$ for all positive braid knots with non-maximal signature invariant~\cite{baader}. In this paper, we exhibit a quantitative gap between $g(K)$ and $\widehat{g}_t(K)$ in the case of alternating non-positive knots. Let us point out that this difference is zero for positive alternating knots, since these have maximal signature invariant.
In fact, the result holds more generally for homogeneous knots, a class of knots introduced by Cromwell \cite{cromwell} that includes alternating knots. 
Namely, a knot is called \emph{homogeneous} if it admits a homogeneous knot diagram, i.e. a diagram in which for every Seifert circle $S$, all crossings adjacent to $S$ inside of $S$ have the same sign,
as do all crossings adjacent to $S$ outside of $S$.
\begin{theorem}\label{thm:main}
Let $K$ be a prime homogeneous knot, and suppose neither $K$ nor its mirror image is a positive knot.
Then $g(K) - \widehat{g_t}(K) \geq 1/3$.
\end{theorem}
We do not know whether an equivalent statement holds if $g_t(K)$ is replaced by the smooth slice genus $g_s(K)$, which is defined by minimising the genus over all smooth surfaces $\Sigma \subset B^4$ with boundary $\partial \Sigma=K$.
However, we are still able to prove the following qualitative version.
\begin{theorem}\label{thm:smooth}
Let $K$ be a homogeneous (not necessarily prime) knot,
and suppose neither $K$ nor its mirror image is a positive knot.
Then $g(K) - \widehat{g_s}(K) > 0$.
\end{theorem}
Livingston constructed families of (non-homogeneous) knots with arbitrarily small positive difference $g - \widehat{g_s}$.
We do not know if such a family exists in the topological setting.

The proof of \cref{thm:main} occupies the main part of this paper.
It is based on Freedman's disc theorem, a version of the Hasse principle
and an elementary fact about surface mapping classes.

\emph{Acknowledgements:} We thank Thomas Huber for showing us how to find zeros of indefinite quadratic forms in few variables.
We also thank Peter Feller for comments on a first version of this paper.

\section{Proofs of the main theorems}
The proof of \cref{thm:main} is split into three lemmas.
\begin{lemma}\label{lem:hom}
Let $D$ be a homogeneous reduced diagram of a knot $K$ with positive and negative crossings,
and $\Sigma$ the Seifert surface obtained by Seifert's algorithm.
\begin{enumerate}\renewcommand{\theenumi}{\roman{enumi}}
\item The surface $\Sigma$ contains two unknotted simple closed curves $\gamma_{\pm}$
of positive and negative induced framing, respectively.
\item If $D$ is prime, $\Sigma$ contains two unknotted  simple closed curves $\gamma_{\pm}$
such that the Seifert form on $\langle [\gamma_+],[\gamma_-]\rangle$ has the following matrix with $p, n > 0$:
\[
\begin{pmatrix}
p & 1 \\
0 & -n
\end{pmatrix}.
\]
\end{enumerate}
\end{lemma}
\begin{proof}
The Seifert surface $\Sigma$ consists of Seifert discs and twisted bands,
the latter being in one-to-one correspondence with the crossings of $D$.
Let $\Gamma\subset\Sigma$ be the union of the boundaries of the Seifert discs and the cores of the twisted bands.
Note that $\Gamma$ is a trivalent graph embedded into $\Sigma \subset\mathbb{R}^3$.
Let $\pi:\mathbb{R}^3\to\mathbb{R}^2$ be the projection that sends $K$ to $D$.
Then $\pi|_{\Gamma}$ is a homeomorphism onto its image, i.e. $\pi|_{\Gamma}$ is a drawing
of $\Gamma$ as the plane graph $\pi(\Gamma)$.
Let $U$ be a connected component of $\mathbb{R}^2\setminus \pi(\Gamma)$.
Since $D$ is reduced, $\partial U$ is a simple closed curve in $\pi(\Gamma)\subset\mathbb{R}^2$.
It lifts to a simple closed curve in $\Gamma\subset\Sigma$, which we denote
by $\gamma_U$. Since $\pi(\gamma_U)$ is a crossingless knot diagram of $\gamma_U$,
this curve is unknotted.
\begin{figure}[b]
\hspace{1cm}
\parbox[b]{.45\textwidth}{%
(a) 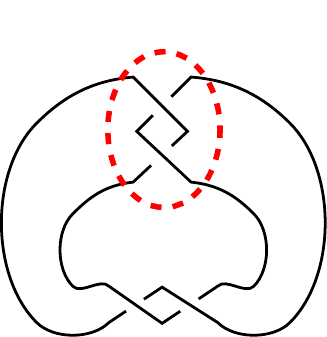 \\[7ex]
(b) 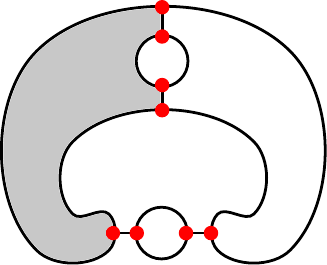}
\hfill
\parbox[b]{.45\textwidth}{%
(c) \hspace{-90pt}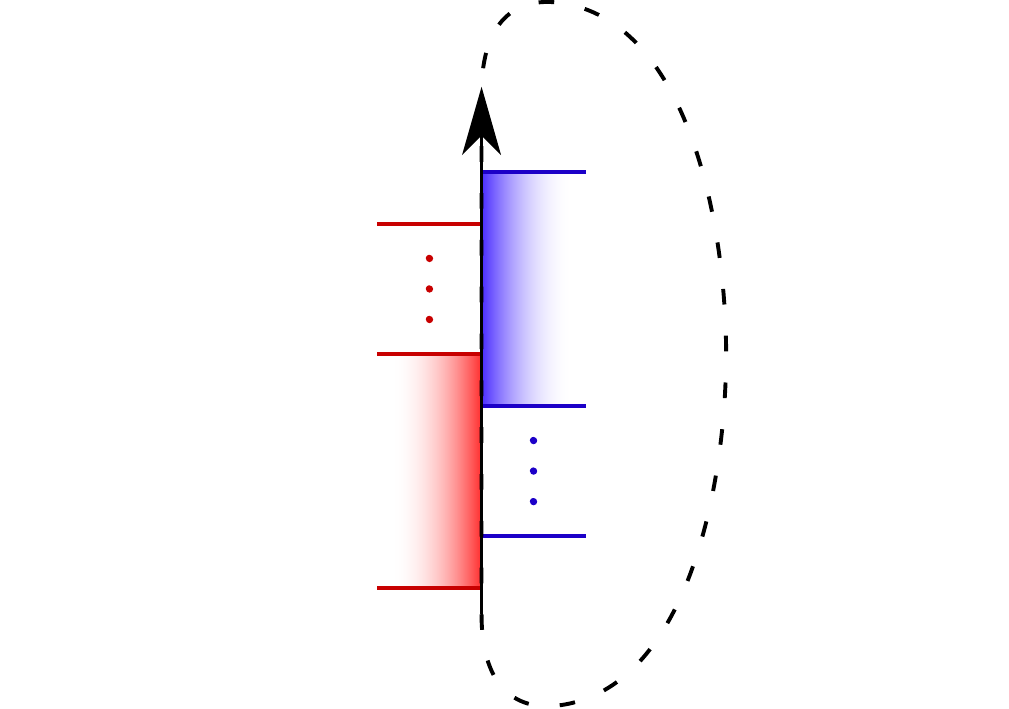}\hfill\mbox{}\\
\caption{
\emph{(a)} A diagram of the figure-eight knot.
\emph{(b)} The graph $\pi(\Gamma)$ with some component $U$ of $\mathbb{R}^2\setminus\pi(\Gamma)$ highlighted in grey.
The two crossings contained in $C(U)$ are circled in the knot diagram shown in (a).
\emph{(c)} The construction used in the proof of \cref{lem:hom}(ii).}
\label{fig:yeah}
\end{figure}

Let $C(U) \subset D$ be the set of those crossings $c$ of $D$
for which $\partial U$ contains the edge of $\pi(\Gamma)$ corresponding to $c$;
or, equivalently, of those crossings $c$ for which $\gamma_U$ passes through the twisted band corresponding to $c$.
An annulus around $\gamma_U$ in $\Sigma$
twists whenever $\gamma_U$ passes through a twisted band.
So  the induced framing of $\gamma_U$ in $\Sigma$ equals
\[
\sum_{c\in C(U)} \sgn(c).
\]
See \Cref{fig:yeah} for an example.

To construct a curve $\gamma_+$ with positive induced framing for part (i) of the lemma,
pick an edge in $\Gamma$ corresponding to a positive crossing.
Let $U$ be a connected component of $\mathbb{R}^2\setminus \pi(\Gamma)$ whose boundary contains this edge.
Since $D$ is homogeneous, all crossings contained in $C(U)$ are of the same sign.
Therefore, $\gamma_+ = \gamma_U$ has positive induced framing.
The curve $\gamma_-$ may be constructed analogously.

Now let us assume that $D$ is prime and prove part (ii).
Because $D$ has crossings of both signs and is connected,
there is a Seifert circle $S$ such that the set $C(S)$ of crossings
adjacent to $S$ contains positive and negative crossings.
The orientation of $S$ endows $C(S)$ with a cyclic ordering.
Let $c_1\in C(S)$ be a positive crossing with negative successor (with respect to that ordering).
Let $c_2$ be the first crossing after $c_1$ with positive successor $c_3$.
Let $c_4$ be the first negative crossing after $c_3$. \cref{fig:yeah}(b) illustrates the choice of the $c_i$.

We shall now prove that $c_1, c_2, c_3, c_4$ are four distinct crossings
that lie on $S$ in that order.
Assume the contrary.
Then there is only one block of positive and negative crossings of $C(S)$ each,
which contradicts the primality of $D$; more precisely, pick $x,y\in S\setminus C(S)$,
namely $x$ between $c_1$ and its successor, and $y$ between $c_3$ and its predecessor.
Then $S\setminus\{x,y\}$ splits into
two components, one containing all positive,
the other containing all negative crossings of $C(S)$.
The complement of $S$ in $\mathbb{R}^2$ consists of two connected components.
In each of them, $x$ and $y$ may be connected by a path that avoids $D$.
The union of these two paths is a circle whose intersection with $D$ is $\{x,y\}$,
and which contains crossings of $D$ on either side; in other words, $D$ is not prime.

Now let $U_+$ be the connected component of $\mathbb{R}^2\setminus \pi(\Gamma)$
that lies vis-\`{a}-vis of $c_2$, on the other side of $S$, and similarly $U_-$
the component that lies vis-\`{a}-vis of $c_3$.
Let $\gamma_+ = \gamma_{U_+}$ and $\gamma_- = \gamma_{U_-}$.
Then the diagonal entries of the matrix of the Seifert form of $\Sigma$ on $\langle[\gamma_+],[\gamma_-]\rangle$ have the desired form,
because the associated curves  have positive and negative induced framing, respectively.
Furthermore, after a small generic isotopy, the curves have one geometric intersection point, whereas the projections
$\pi(\gamma_{\pm}) \subset \mathbb{R}^2$ intersect twice. This implies that one of the off-diagonal entries
vanishes, and the other one is $\pm 1$.
Switching the orientations of $\gamma_+$ and of $\Sigma$ if necessary establishes the desired matrix.
\end{proof}
\begin{lemma}\label{lem:numbertheory}
Let $p, n > 0$ be given. Then there are $x_1, x_2, y_1, y_2 \in \mathbb{Z}$ such that
\begin{equation}\tag{$\dagger$}\label{eq:1}
\sum_{i=1}^2 px_i^2 + x_iy_i - ny_i^2 = -p.
\end{equation}
\end{lemma}
\begin{proof}
First, one may complete the square and rewrite the summands as follows:
\[
px_i^2 + x_iy_i - ny_i^2  =
p\left(\left(x_i + \frac{y_i}{2p}\right)^2 -
(1 + 4np)\left(\frac{y_i}{2p}\right)^2\right).
\]
Let us substitute $\bar{y}_i = y_i/2p$ and $\bar{x}_i = x_i + \bar{y}_i$.
Then
\begin{align}
\notag \sum_{i=1}^2 px_i^2 + x_iy_i - ny_i^2 & = -p  \quad\Leftrightarrow\\
\tag{$\ddagger$} \bar{x}_1^2 + \bar{x}_2^2 - (1+4np)(\bar{y}_1^2 + \bar{y}_2^2) & = -1. \label{eq:2}
\end{align}
Solving \cref{eq:2} in $\bar{x}_i, \bar{y}_i$ will give an integer solution to \cref{eq:1}.
The equation \Cref{eq:2} is quadratic in four variables, so by a version of 
the Hasse principle due to Watson \cite{watson},
the existence of an integral solution is equivalent to the
existence of a real solution and a solution in $\mathbb{Z}/m\mathbb{Z}$
for every positive $m$. The real solution exists because the associated
symmetric bilinear form is indefinite.
By the Chinese Remainder Theorem it is sufficient to find solutions in
$\mathbb{Z}/m\mathbb{Z}$ for $m = q^k$ a prime power.

For $k=1$, i.e. $m$ prime, it is well-known that every number in $\mathbb{Z}/m\mathbb{Z}$
is a sum of two squares, so there is even a solution to \cref{eq:2} with $\bar{y}_i = 0$.
From there, if $q$ is an odd prime, we can inductively construct a solution in $\mathbb{Z}/q^k\mathbb{Z}$ for $k\geq 2$:
\begin{align*}
\bar{x}_1^2 + \bar{x}_2^2 & \equiv -1 \pmod{q^{k-1}} \quad\Rightarrow \\
\exists\alpha\in\mathbb{Z}:
\bar{x}_1^2 + \bar{x}_2^2 & \equiv -1 + \alpha q^{k-1} \pmod{q^k} \quad\Rightarrow \\
\forall\beta\in\mathbb{Z}:
\bar{x}_1^2 + (\bar{x}_2+\beta q^{k-1})^2 & \equiv -1 + (2\beta\bar{x}_2 + \alpha) q^{k-1} \pmod{q^k}
\end{align*}
Now suppose w.l.o.g. that $\bar{x}_2\not\equiv 0\pmod{q}$, then
taking $\beta$ to be $-\alpha/(2\bar{x}_2)\in\mathbb{F}_q$ gives a solution.

For $q= 2$ and $k \leq 3$,
a solution is given by $(\bar{x}_1, \bar{x}_2, \bar{y}_1, \bar{y}_2) = (1,0,1,1)$.
For $k \geq 4$, one can inductively construct solutions (very much as above)
with odd $\bar{x}_1$. Given such a solution
\[
\bar{x}_1^2 + \bar{x}_2^2 - (1+4np)(\bar{y}_1^2 + \bar{y}_2^2)  \equiv -1 \pmod{2^{k-1}},
\]
we either have
\[
\bar{x}_1^2 + \bar{x}_2^2 - (1+4np)(\bar{y}_1^2 + \bar{y}_2^2)  \equiv -1 \pmod{2^k},
\]
and are done; or we have
\begin{align*}
\bar{x}_1^2 + \bar{x}_2^2 - (1+4np)(\bar{y}_1^2 + \bar{y}_2^2) & \equiv 2^{k-1}-1 \pmod{2^k} \quad\Rightarrow \\
(\bar{x}_1 + 2^{k-2})^2 + \bar{x}_2^2 - (1+4np)(\bar{y}_1^2 + \bar{y}_2^2) & \equiv -1 \pmod{2^k}
\end{align*}
since $\bar{x}_1$ is odd.
\end{proof}
\begin{lemma}\label{lem:top}
Let $A$ be the Seifert form of a genus-minimal Seifert surface $\Sigma$
of a knot $K$.
\begin{enumerate}\renewcommand{\theenumi}{\roman{enumi}}
\item If there is a subgroup of rank two of $H_1(\Sigma)$
on which the matrix of $A$ is
\[
\begin{pmatrix}
0 & 1 \\
0 & x
\end{pmatrix}
\]
for some integer $x$, then $g_t(K) \leq g(K) - 1$.
\item If there is a subgroup of rank two of $H_1(\Sigma)$
on which the matrix of $A$ is
\[
\begin{pmatrix}
p & 1 \\
0 & -n
\end{pmatrix}
\]
with $p, n \geq 0$, then $\widehat{g_t}(K) \leq g(K) - 1/3$.
\end{enumerate}
\end{lemma}
\begin{proof}
\textbf{(i)}
Let $(v, w)$ be a basis of the subgroup in question
with respect to which $A$ has the given matrix.
By subtracting $x\cdot v$ from $w$, we can assume that $x=0$.
We now need to show the existence of simple closed curves $\gamma, \zeta\subset \Sigma$
with $[\gamma] = v, [\zeta] = w$ and $|\gamma\cap \zeta| = 1$.
Let $\Sigma'$ be an (abstract) compact oriented surface of genus $g$ with one boundary component,
and $\gamma', \zeta'$ simple closed curves on $\Sigma'$ with $|\gamma'\cap \zeta'| = 1$.
Then there is an isomorphism $\varphi: H_1(\Sigma') \to H_1(\Sigma)$ respecting the intersection forms
with $\varphi([\gamma']) = v, \varphi([\zeta']) = w$.
Since the mapping class group surjects onto the symplectic group (see e.g. \cite{primer}),
$\varphi$ is realised as the action of a homeomorphism $\widetilde{\varphi}: \Sigma'\to\Sigma$,
and one may take $\gamma=\widetilde{\varphi}(\gamma'), \zeta=\widetilde{\varphi}(\zeta')$.

Let $T$ be the union of closed tubular neighbourhoods of $\gamma$ and $\zeta$ in $\Sigma$, i.e.
the union of two annuli whose cores intersect once.
Then $T$ is a surface of genus one with one boundary component.
Since $T$ is a subsurface of $\Sigma$, the Seifert form of $T$ is simply the restriction of $A$ to $H_1(T)$,
and so with respect to the basis $([\gamma], [\zeta])$ of $H_1(T)$, its matrix is 
\[
\begin{pmatrix}
0 & 1 \\
0 & 0
\end{pmatrix}.
\]
Therefore the Alexander polynomial of $\partial T$ is trivial, and thus by Freedman's disc theorem~\cite{freedman2,freedman}
there is a locally flat embedding of a disc $D$ into $B^4$ with $\partial D = \partial T = D \cap S^3$.
Replace $T$ with $D$, i.e. remove $T$ from $\Sigma$ and glue $D$ to the resulting boundary component. This yields
a connected locally flat slice surface $S$ of $K$ with genus one less than the genus of $\Sigma$.\\

\textbf{(ii)}
If $p=0$ or $n=0$, the hypothesis of part (i) of this Lemma
is satisfied, and we even have $\widehat{g_t}(K) \leq g_t(K) \leq g(K) - 1$.
So let us assume $p, n > 0$.
The idea is to apply part (i) to the knot $K^3$ with genus-minimal Seifert surface $\Sigma^3 = \Sigma \# \Sigma \# \Sigma$.
Fix the basis of the rank-two subgroup of $H_1(\Sigma)$ with respect to which $A$ has the given matrix.
Three copies of this basis concatenated give a basis of a rank-six subgroup $V$ of
$H_1(\Sigma)^{\oplus 3} \cong H_1(\Sigma^3)$.
By \cref{lem:numbertheory}, there are $x_1, x_2, y_1, y_2\in\mathbb{Z}$ such that
\[
\sum_{i=1}^2 px_i^2 + x_iy_i - ny_i^2 = -p.
\]
Consider the rank-two subgroup of $V$ generated by the vectors
$(1,0,x_1,y_1,x_2,y_2)^{\top}$ and $(0,1,0,0,0,0)^{\top}$.
On this subgroup, $A$ has the matrix given in part (i). Hence
\[
g_t(K^3) \leq g(K^3) - 1 \quad\Rightarrow\quad
\widehat{g_t}(K) \leq g(K) - 1/3.
\]
\end{proof}
\Cref{thm:main} follows from \Cref{lem:hom}(ii) and \cref{lem:top}(ii),
since the canonical Seifert surface of a homogeneous diagram is genus-minimal \cite{cromwell}.
In the smooth setting, \cref{thm:smooth} follows from \cref{lem:hom}(i) and the following:
\begin{lemma}
Let $\Sigma$ be a genus-minimal Seifert surface of a knot $K$.
\begin{enumerate}\renewcommand{\theenumi}{\roman{enumi}}
\item If $\Sigma$ contains a non-separating unknotted simple closed curve $\gamma$ with induced framing $0$, then
\[
g_s(K) \leq g(K) -1.
\]
\item If $\Sigma$ contains two
unknotted curves $\gamma_{\pm}$ with respective framings $p$ and $-n$, where $p, n> 0$, then
\[
\widehat{g_s}(K) < g(K).
\]
\end{enumerate}
\end{lemma}
\begin{proof}
\textbf{(i)} This is just an application of ambient surgery:
remove a tubular neighbourhood of $\gamma$ in $\Sigma$, and glue in two copies of a slice disc of the curve $\gamma$.\\
\textbf{(ii)} In $\Sigma^{n+p}$, one finds the unknotted curve $\gamma_+^n \# \gamma_-^p$
with framing $0$. By part (i) of this Lemma,
\[
g_s(K^{n+p}) \leq g(K^{n+p}) - 1 \quad\Rightarrow\quad
\widehat{g_s}(K) \leq g(K) - \frac{1}{n+p}.
\]
\end{proof}

\section{Applications and examples}
\begin{prop}
Let $K$ be a knot with a Seifert surface of genus one that contains a simple
closed curve of induced framing $1$.
If the signature of $K$ vanishes, then $\widehat{g_t}(K) \leq 2/3$.
\end{prop}
\begin{proof}
Extend the curve's homology class $v$ to a basis $(v,w)$.
With respect to that basis the Seifert form $A$ has the matrix
\[
\begin{pmatrix}
1 & x+1 \\
x & *
\end{pmatrix}
\]
for some integer $x$. Change to the basis $(v, w-x\cdot v)$, on which $A$ has the matrix
\[
\begin{pmatrix}
1 & 1 \\
0 & -n
\end{pmatrix}
\]
for some integer $n$. Since the signature vanishes, we have $n \geq 0$.
Now apply \cref{lem:top}(ii).
\end{proof}
The following is a strengthening of \cite[Corollary 1]{baader}.
\begin{prop}
Let $K$ be a knot with a genus-minimal Seifert surface
that contains simple closed curves of induced framings $1$ and $-1$.
Then $g(K) - \widehat{g_t}(K) \geq 1/2$.
\end{prop}
\begin{proof}
Let $\gamma$ be a curve of induced framing $1$. Let $\zeta$ be
another curve that intersects $\gamma$ once. Then, for some $x,y\in\mathbb{Z}$,
the Seifert form of the Seifert surface $\Sigma$ has the following matrix on $\langle [\gamma], [\zeta]\rangle$:
\[
\begin{pmatrix}
1 & y + 1 \\
y & x
\end{pmatrix}.
\]
Consider $\Sigma\# \Sigma$, and let $\eta$ be a curve of induced framing $-1$ in the second copy of $\Sigma$.
Then the Seifert form has the following matrix on the subgroup $\langle [\gamma] + [\eta], [\zeta] - y\cdot [\gamma]\rangle$:
\[
\begin{pmatrix}
0 & 1 \\
0 & *
\end{pmatrix},
\]
and we are done by \cref{lem:top}(i).
\end{proof}
Note that in both of the previous propositions, the hypotheses can be phrased
in purely homological terms, without reference to curves.
\begin{example}
Let $K_n$ denote the $n$-twist knot, i.e. the two-bridge knot corresponding to the fraction $(4n+1)/2$.
In Rolfsen's table, $K_1 = 4_1, K_2 = 6_1$ etc.
The knot $K_n$ is topologically slice if and only if $n\in \{0,2\}$ \cite{cg},
and for all other $n$ we have $g(K_{n}) = g_t(K_{n}) = 1$.
On the other hand, our main result shows that $\widehat{g_t}(K_{n}) \leq 2/3$ for all $n\geq 0$.

Let us have a closer look at $K_5$. Let $\Sigma$ be its standard genus one Seifert surface.
In what follows, we show that $\Sigma\# \Sigma$ does not contain a simple closed curve $\gamma$
of induced framing $0$.
This demonstrates that while it might a priori well be possible to improve the constant of $1/3$ in \cref{thm:main},
possibly even up to $1$, this would require a new approach.

Taylor \cite{taylor} shows that if a knot has a genus-minimal Seifert surface without
a curve of induced framing $0$, then the genus and smooth slice genus of that knot are equal.
It is generally believed that this results hold in the topological category as well.
This would imply $g_t(K_5\# K_5) = 2$.

The Seifert form of $\Sigma$ has a matrix
\[
\begin{pmatrix}
1 & 1 \\
0 & -5
\end{pmatrix}.
\]
So the homology class of a simple closed curve $\gamma\subset\Sigma\#\Sigma$
with induced framing 0 is a non-trivial solution $(x_1, y_1, x_2, y_2) \in\mathbb{Z}^4$
to the equation
\[
\sum_{i=1}^2 x_i^2 + x_iy_i - 5y_i^2 = 0.
\]
Similarly as in the proof of \cref{lem:numbertheory}, this leads
to a non-trivial solution $\bar{x}_i, \bar{y}_i\in\mathbb{Z}$ of the equation
\[
 \bar{x}_1^2 + \bar{x}_2^2    = 21(\bar{y}_1^2 + \bar{y}_2^2).
 \]
According to a theorem of Fermat, the prime factorisation of a sum of two squares
contains primes that are equivalent to $3\pmod{4}$ an even number of times.
Hence the maximal power of $3$ that divides the left hand side of the above equation
is even, and so is the maximal power of $3$ dividing $(\bar{y}_1^2 + \bar{y}_2^2)$;
thus the maximal power of $3$ dividing the right hand side is odd, which is a contradiction.
\end{example}
\begin{example}
One finds the following alternating non-positive prime knots of genus two
for which Taylor's obstruction (as in the previous example) implies that the topological slice genus equals the genus:
\[
12a_{787}, 13a_{4857}, 14a_{8117}, 14a_{9067}, 15a_{35895}, 15a_{35937}, 15a_{51786}, 15a_{76589}
\]
For these knots, our main result shows that the stable topological slice genus
is at most $5/3$.

It seems probable that similar examples exist for genus three or higher.
However, there are no workable tools at our disposition to prove $g_t = g$ for knots with
non-maximal signature and $g \geq 3$.
\end{example}

\bibliographystyle{myamsalpha}
\bibliography{stable_alt}
\end{document}